\documentclass{amsart}

\usepackage{amsmath,amssymb,amscd,amsfonts}
\usepackage{epsfig, graphicx,pinlabel,pdfsync}
\usepackage[all]{xy}

\newtheorem{thm}{Theorem}[section]
\newtheorem{lem}[thm]{Lemma}
\newtheorem{prop}[thm]{Proposition}
\newtheorem{coro}[thm]{Corollary}
\newtheorem*{theoremquote}{Theorem}

\theoremstyle{definition}
\newtheorem{defn}[thm]{Definition}
\newtheorem{prop-defn}[thm]{Proposition-Definition}
\newtheorem{exmp}[thm]{Example}

\newtheorem*{defnquote}{Definition}
\newtheorem*{prop-defn-quote}{Proposition-Definition}

\theoremstyle{remark}
\newtheorem{remark}[thm]{Remark}

\numberwithin{equation}{section}

\begin{document}

\title{Generic states and stability}
\date{\today}
\author{Donghoon Hyeon and Junyoung Park}

\address[DH]{Department of Mathematical Sciences, Seoul National University, Seoul, R. O. Korea \\Tel: +82-2-880-2666, Fax: +82-2-887-4694 }
\email{dhyeon@snu.ac.kr}
\address[JP]{ Department of Mathematics, POSTECH, Pohang, Gyungbuk, R. O. Korea}
\email{newshake@postech.ac.kr}

\keywords{Geometric Invariant Theory \and state polytope }
\subjclass[2010]{14L24}

\def\til{\widetilde}
\textbf{}
\def\ten{\otimes}
\def\ex{\times}

\def\ul{\underline}
\def\ol{\overline}

\def\a{\alpha}
\def\b{\beta}
\def\d{\delta}
\def\D{\Delta}
\def\e{\epsilon}
\def\g{\gamma}
\def\Gm{\Gamma}
\def\G{\Gamma}
\def\D{\Delta}
\def\h{\eta}

\def\la{\lambda}
\def\m{\mu}
\def\n{\nu}
\def\sm{\sigma}
\def\Sm{\Sigma}
\def\o{\omega}
\def\Om{\Omega}
\def\p{\phi}
\def\tt{\theta}
\def\w{\wedge}
\def\W{\bigwedge}
\def\vphi{\varphi}
\def\vp{\varphi}
\def\A{\mbb A}
\def\C{\mbb C}
\def\bR{\mbb R}
\def\bP{\mbb P}
\def\bZ{\mbb Z}
\def\bN{\mbb N}
\def\bQ{\mbb Q}
\def\t{\tau}
\def\cZ{\mathcal Z}
\def\cJ{\mathcal J}
\def\cI{\mathcal I}
\def\SL{{\mathrm {SL}}}
\def\GL{{\mathrm {GL}}}
\def\PSL{{\mathrm {PSL}}}

\def\bE{\mathbb E}
\def\bP{\mathbb P}
\def\bG{\mathbb G}
\def\bC{\mathbb C}
\def\bZ{\mathbb Z}
\def\bR{\mathbb R}
\def\inj{\hookrightarrow}
\def\surj{\twoheadrightarrow}
\def\dra{\dashrightarrow}


\def\cO{\mathcal O}
\def\deg{\textup{deg} \, }
\def\rk{\textup{rank} \, }
\def\lra{\longrightarrow}
\def\lla{\longleftarrow}
\def\rd{\partial}
\def\inv{^{-1}}
\def\isom{\simeq}
\def\der{\mathrm{Der\,}}
\def\spec{\mathrm{Spec \,}}
\def\proj{{\rm Proj \, }}
\def\ord{\mathrm{ord\,}}
\def\Proj{\rm{\textbf{Proj} \,}}
\def\aut{\textup{Aut}}

\def\mod{/ \! \! /}

\def\inj{\hookrightarrow}

\def\Pic{\mathrm{Pic}}
\def\PGL{\mathrm{PGL}}
\def\Sym{\mathrm{Sym}}
\def\Hom{\mathrm{Hom}}
\def\Gr{\textup{Gr}}

\def\cO{\mathcal O}
\def\ra{\rightarrow}


\def\bar{\overline}
\def\isomto{\stackrel{\sim}{\to}}
\def\codim{\textup{codim}\,}
\def\st{\Xi}
\def\stp{\mathcal{P}}
\def\gst{\mathcal{G}\Xi}
\def\gstp{\mathcal{GP}}
\def\conv{\mathrm{Conv} \,}
\def\gin{\mathrm{Gin}}
\def\weyl{\mathcal W}

\input xy
\xyoption{all}

\input epsf
\epsfxsize=2in

\def\bQ{\mathbb Q}
\def\tilde{\widetilde}
\def\hilb{\textup{Hilb}\,}

\def\Mg{\bar M_g}
\def\M{\bar M}
\def\PGL{\textup{PGL}}
\def\ii{{\bf in}\,}

\thanks{The first author was supported by the Research Resettlement Fund for the new faculty of Seoul National University, the SNU Invitation Program for Distinguished Scholar, and the following grants funded by the government of Korea:
NRF grant 2011-0030044 (SRC-GAIA) and NRF grant NRF-2013R1A1A2010649.}

\maketitle

\begin{abstract}  We define  the notion of the generic state polytope, analogous to the generic initial ideal and prove its existence: This greatly generalizes the work of R\"omer and Schmitz who proved the existence of generic Gr\"ober fans.  We also  show that a generic state polytope always contains the trivial character: Equivalently, in any GIT quotient problem of semisimple group representations, every point is semistable with respect to a {\it general} maximal torus. Also, we revisit Kempf's proof of the existence of the worst one parameter subgroup (1-ps) and describe the equations for determining the worst 1-ps.
\end{abstract}

\section{Introduction \& Preliminaries} Let $G$ be a reductive group over an algebraically closed field $k$ of characteristic zero and let $V$ be a rational representation of $G$.  A {\it state} of the $G$ action on $V$ is a function $\Xi$ that assigns to each torus $R$ of $G$ a subset $\Xi(R) \subset X(R)$ such that for any two tori $R_1 \subset R_2$, the image of $\Xi(R_2)$ under $X(R_2) \to X(R_1)$ is $\Xi(R_1)$ \cite{Kempf}.  The {\it state polytope} of $\Xi$ with respect to a torus $R$  is defined to be the convex hull in $X(R)_{\mathbb R}$ of $\Xi(R)$. A state of particular interest is the {\it Geometric Invariant Theory (GIT) state} of a point $v \in V$:
\[
\Xi_v(R)  := \{ \xi \in X(R) \, | \, v_\xi \ne 0\}.
\]
Here, $v = \sum v_\xi$ is the $R$-weight decomposition of $v$. For instance, if we take $G = \GL(V)$ acting naturally on $\bigwedge^{P(m)}S^mV^*$, we retrieve the state polytopes of the $m$th Hilbert points of graded ideals of the graded algebra $\oplus_{m\ge 0}S^mV^*$ with Hilbert polynomial $P$ (\cite{BM}, \cite{Sturmfels}). In this case, Bayer and Morrison proved that the state polytope can be computed via Gr\"obner basis as follows:
\begin{theoremquote}\cite[Theorem~3.1]{BM}\label{T:BM} \, For $m$ sufficiently large, there is a natural one-to-one correspondence between saturations of initial ideals and the vertices of the state polytope.
 \end{theoremquote}

For the rest of this article, a state will always mean a GIT state, and the state polytope of $v$ with respect to $R$ will be conveniently denoted by $\stp_v(R)$. A GIT state is {\it bounded} in the sense that for every torus $R$, $\bigcup_{g\in G} g_{!}(\Xi_v(g^{-1}Rg))$ is a finite subset of $X(R)$, where $g_!$ is the isomorphism $X(g^{-1}Rg) \isomto X(R)$ given by the conjugation by $g$.
This is one of the key notions that leads to Kempf's proof (Theorem~2.2, ibid) of the existence of a worst one parameter subgroup conjectured by Mumford. A natural question arises: {\it Given algebraic group and an unstable point in a rational representation, how would one actually compute a worst 1-ps?}

Once we fix a maximal torus $R$, it can be relatively easily computed, as it corresponds to the nearest point of the state polytope (see Section~\ref{S:algorithm}, Theorem~\ref{T:nearest}, for a brief account of this fact which is well known in approximation theory). But to find the worst 1-ps or to determine whether the given point is semistable, we need to do this for every maximal tori. Indeed, the difficult part is to understand how the state polytopes vary according to the group action, and we try to do just that in this article.

Our first theorem in this direction is that the state polytope is invariant under the translation by a generic element of the group. This is similar to the notion of the {\it generic initial ideal}:
Recall, for instance from \cite[Section~1.59]{Eisenbud}, that for a given ideal $I \subset S=k[x_0, \dots, x_n]$ and a monomial order $\prec$, there exists a generic initial ideal $\gin_\prec(I)$ such that $\ii_\prec(g.I) = \gin_\prec(I)$ for any $g$ in an open set $U$ of $\GL_{n+1}(k)$ acting naturally on $S$. Our result is a direct generalization to state polytopes:

\begin{prop-defn-quote}\label{D:gst} Let $G$ be an algebraic group, $V$ be a rational representation of $G$, and $v \in V$.  There exists a nonempty open subscheme $U \subset G$ such that for every $g, g'\in U$, $\stp_{g.v}(R) = \stp_{g'.v}(R)$. The {\it (primary) generic state polytope} $\gstp_v(R)$ of $v$ with respect to $R$ is the state polytope $\stp_{g.v}(R)$ for some (and hence for any) $g \in U$.
\end{prop-defn-quote}
 The existence of such an open set $U$ in the above definition will be established in Proposition~\ref{P:gstv}.  As an immediate corollary,  we retrieve the existence of generic Gr\"ober fan proved in \cite[Theorem~3.1, Corollary~3.2]{Romer} by analyzing the behavior of universal Gr\"obner basis under linear coordinate changes.

 \

 In fact, we prove  more than just the existence of the open set $U$ and $\gstp_v(R)$.  In turn, there is an open set of each component  of the closed subscheme $G \setminus U$ on which the state polytope remains the same. Indeed, Proposition~\ref{P:gstv} asserts that there is a {\it finite} stratification $G = \coprod_{S \in I} U_S^v$  such that the state polytope $\stp_{g.v}(R)$ remains unchanged on each locally closed stratum $U_S^v \subset G$. This means that to find a worst 1-ps, one would compute the {\it finitely many} state polytopes corresponding to the finite stratification, find the nearest point of each state polytope thus computed, and compare them to find the farthest point among all of them (Section~\ref{S:algorithm}). That is, the state polytope that is farthest away from the origin gives the worst 1-PS.

\

From the discussion above, it is clear that the worst one parameter subgroups are contained in very special maximal tori. We turn our attention to a more general problem: When $G$ is reductive, a maximal torus $R$ of $G$ contains a  1-ps destabilizing $v$  if and only if $\stp_v(R)$ does not contain the trivial character ($=$ the origin). We know from experience that even in the worst GIT problems, for instance even when {\it every} point is unstable, destabilizing one-parameter subgroups, worst or not, are pretty special and a torus must be  chosen judiciously to destabilize a point. Especially, our experience is that every Hilbert point is stable with respect to \emph{generic} coordinates. To make this precise in terms of state polytopes, we make the following definition:

\begin{defnquote} Let $G$ be an algebraic group and $V$ be a rational representation of $G$. A point $v \in V$ is \emph{generically semistable (resp. stable)} if the primary generic state polytope $\gstp_v(R)$  (resp. the interior of $\gstp_v(R)$) contains the origin.
\end{defnquote}
In any GIT problem of Hilbert schemes, mainly because  equations in generic co-ordinates should bear enough symmetry with respect to the variables, Hilbert points are generically stable (Proposition~\ref{P:gen-stab-hilb}).
In general, we have the following main theorem:

\begin{theoremquote} Let $G$ be a semisimple algebraic group and $V$ be a non-trivial rational representation. Then
\begin{enumerate}
\item every nonzero $v \in V$ is generically semistable, and;
\item it is generically stable if and only if its isotropy subgroup $G_v$ contains no nontrivial almost simple factor of $G$.
\end{enumerate}
\end{theoremquote}

A geometric formulation of the theorem is as follows. Note that an isotropy group $G_v$ contains a nontrivial almost simple factor of $G$ if and only if all isotropy subgroups $G_{g.v} = g G_v g$ do so. This is what we mean  in (2) below  where we talk about ``the isotropy subgroup of the orbit defined up to conjugacy''.

\begin{theoremquote} Let $G$ be a semisimple algebraic group acting linearly on $X = \bP(V)$. Let $R \subset G$ be a maximal torus.
\begin{enumerate}
\item Any $G$-orbit in $X$ intersects the locus $X^{ss}(L,R)$ of semistable points with respect to the $R$ action and the $R$-linearized line bundle $L=\cO_{\bP(V)}(+1)$.
\item A $G$-orbit in $X$ intersects the locus $X^s(L,R)$ of stable points if and only if the isotropy subgroup of the orbit (defined up to conjugacy) contains no nontrivial almost simple factor of $G$.
\end{enumerate}
\end{theoremquote}

\noindent We shall prove our main theorem  in Section~\ref{S:generic-semistability}.

\

A corresponding statement for the case when $G$ is a reductive group can be easily deduced from our main theorem above by using the relation between the state polytope of $G$ and the state polytope of the semisimple derived group $[G,G]$. After this work was completed, the first named author and Dao Phoung Bac carried this out, and we state the result here for completeness and for the convenience of the readers.

\begin{theoremquote}\cite[Theorem~5.1]{BH} Let
$G$ be a reductive group and let $V$ be a rational representation of $G$. Then $v\in V$ is generically semistable if and only if it is semistable with respect to the radical $\mathcal R(G)$ of $G$.
\end{theoremquote}

\

We work over an algebraically closed field $k$ of characteristic zero. Henceforth, $G$ will always mean a connected linear algebraic group.



\section{Generic states}\label{S:gst}

Let $\phi : G \to \GL(V)$ be a rational representation and $\sigma : G \times V \to V$ be the corresponding $G$ action on $V$. Given any torus $R \subset V$, we have the weight space decomposition
\[
V = \bigoplus_{\chi \in P(V)} V_\chi
\]
where $V_\chi$ is the weight space corresponding to $\chi \in X(R) = \hom(R, \bG_m)$ i.e. $R$ acts on $v \in V_\chi$ by $t.v = \chi(t)v$, and $P(V) = \{ \chi \in X(R) \, | \, V_\chi \ne 0\}$.

Fix a basis $\{e_{\chi,\a}\}$ for each $V_\chi$ and let $\{f_{\chi,\a}\}$ denote the dual basis.  We shall not index the $\a$'s here as there is no danger of confusion. For each $v \in V$, we write the weight decomposition of $v$ as
\[
v = \sum_{\chi}v_{\chi} = \sum_\chi\left(\sum_\a v_{\chi,\a}e_{\chi,\a}\right) =\sum_\chi \left(\sum_\a f_{\chi,\a}(v) e_{\chi,\a}\right).
\]
Define
\[
h_{\chi,\a} := \sigma^\sharp(f_{\chi,\a}) \in k[G]\otimes_kk[V]
\]
so that $h_{\chi,\a}(g,v) \ne 0$ if and only if $(g.v)_{\chi,\a} \ne 0$. Hence $(g.v)_{\chi} \ne 0$ if and only if $h_{\chi,\a}(g,v)\ne 0$ for some $\a$.
For any  $\chi \in P(V)$, we define
\[
U_\chi := \bigcup_\a (G\times V)_{h_{\chi,\a}}.
\]
This is the locus of $(g,v)$ with $(g.v)_\chi \ne 0$.
Note that for any $\chi \in X(R)$, $V_\chi = 0$ if and only if $h_{\chi,\a} \equiv 0$ in $k[G]\otimes_kk[V]$ for all $\a$. Hence $U_\chi \ne \emptyset$ if and only if $\chi \in P(V)$.

 Let $Z_\chi$ denote the closed subvariety defined by the ideal
\[
I_\chi = \langle h_{\chi,\a} \rangle_\a.
\]
For any subset $S \subset P(V)$, we associate a locally closed subvariety
\[
U_S = \left(\bigcap_{\chi \in S} Z_\chi\right) \bigcap \left(\bigcap_{\chi\not\in S} U_\chi\right).
\]
We remark that $U_S = \emptyset$ if for some $\chi \not\in S$, $h_{\chi,\a} \in \sum_{\chi \in S} I_\chi$ for all $\a$.

It is clear from the definition that
$\Xi_{g.v}(R) = P(V)\setminus S$ if and only if $(g,v) \in U_S$. Also clear is that $U_\emptyset = \bigcap_{\chi \in P(V)} U_\chi$ is Zariski open and dense in $G\times V$.
We summarize our findings into a proposition:

\begin{prop}\label{P:gstV} There exists a decomposition
\[
G\times V = \coprod_{S \subset P(V)} U_S
\]
into locally closed subvarieties such that  $(g,v)$ and $(g',v')$ belong to $U_S$ for some $S$ if and only if
\[
\Xi_{g.v}(R) = \Xi_{g'.v'}(R).
\]
The big open set $U_\emptyset$ is Zariski dense and open in $G\times V$ and $\Xi_{g.v}(R) = P(V)$ for any $(g,v) \in U_\emptyset$.
\end{prop}

Fix $v \in V$ and consider the corresponding GIT state $\Xi_v$. One can give a locally closed decomposition of $G$ in a similar fashion:

\begin{prop}\label{P:gstv} There exists a decomposition
\begin{equation}\label{E:gstv}
G = \coprod_{S\subset P(V)} U^v_S \tag{$\dagger$}
\end{equation}
into locally closed subvarieties such that $\Xi_{g.v}(R) = \Xi_{g'.v}(R)$ if and only if $g$ and $g'$ are in the same stratum of the locally closed decomposition.
There exists a distinguished $S^v \subset P(V)$ such that $U^v_{S^v}$ is Zariski dense open subscheme of $G$.
\end{prop}


\begin{proof}
Let $h^v_{\chi,\a} := \sm_v^\sharp(f_{\chi,\a}) \in k[G]$ where $\sm_v : G \to V$ is the orbit map $g \mapsto g.v$. Put differently, if $\iota_v : G \isomto G\times \{v\} \inj G \times V$ denotes the inclusion, then $h^v_{\chi,\a} = \iota_v^\sharp (h_{\chi,\a})$.

Define $Z^v_\chi$ to be closed subscheme cut out by the ideal $I_\chi = \langle h^v_{\chi,\a} \rangle_\a$, and define
\[
U^v_\chi := G \setminus Z^v_\chi.
\]
Note that $g \in U^v_\chi$ if and only if $(g.v)_\chi \ne 0$. In other words, $U^v_\chi$ equals $U_\chi\cap (G\times \{v\})$ which also follows from that $h^v_{\chi,\a} = \iota_v^\sharp (h_{\chi,\a})$.

For any subset $S$ of $P(V)$, we analogously define
\[
U^v_S = \left(\bigcap_{\chi\in S} Z^v_\chi\right) \bigcap \left(\bigcap_{\chi\not\in S} U^v_\chi\right).
\]
Again, $U^v_S$ is naturally identified with $U_S \cap (G\times \{v\})$, and $\Xi_{g.v}(R) = P(V)\setminus S$ if and only if $g \in U^v_S$. We note here that $U^v_S$ may very well be empty.

Let $S^v := \{\chi \in P(V) \, | \, (g.v)_\chi = 0, \, \forall g\in G\}$, and let $U^v_{S^v} = \bigcap_{\chi\in P(V)\setminus S^v}U^v_\chi$. Since $U^v_\chi\ne\emptyset$ for any $\chi \in P(V)\setminus S^v$,  $U^v_{S^v}$ is a nonempty dense open subscheme of $G$. For any $g \in U^v_{S^v}$, by definition $(g.v)_\chi \ne 0$ for any $\chi \in P(V)\setminus S^v$, which means that $\Xi_{g.v}(R) = P(V)\setminus S^v$ for any $g \in U^v_{S^v}$.
\end{proof}
The proposition above establishes the existence of an open set $U$ in the Proposition-Definition from the introduction.

\begin{defn}\label{D:gstv} The {\it (primary) generic states} of $v$ with respect to $R$ is
	\[
	\gst_v(R) = \Xi_{g.v}(R)
	\]
for any $g \in U^v_{S^v}$.
	\end{defn}
Note that $\gst_v(R) = P(V) \setminus S^v$ where $S^v$ is the distinguished subset of weights defined in the proof of Proposition~\ref{P:gstv}.
\begin{defn} \begin{enumerate}
\item The {\it primary generic state polytope} of $v$ with respect to $R$ is the convex hull in $X(R)_\bR :=X(R)\otimes_\bZ \bR$ of the primary generic states of $v$ with respect to $R$. It is denoted by $\gstp_v(R)$.

\item For any $S \ne \emptyset \subset P(V)$ such that $U_S^v \ne \emptyset$, the convex hull of $\Xi_{g.v}(R)$ in $X(R)_\bR$ for some (and hence for any) $g \in U^v_S$ is called the {\it secondary generic state polytope} of $v$ corresponding to $S$, and is denoted by $\gstp_v(R;S)$.
\end{enumerate}
\end{defn}
Note that $\gstp_v(R;S)$ is just the convex hull of $P(V)\setminus S$.

\section{Generic  Gr\"obner fans}
In this section, we show that the generic Gr\"obner fan \cite{Romer} is virtually a special case of a generic state considered in the previous section.
Let $I$ be a graded ideal of the graded algebra $$\oplus_{m\ge 0}S^mV^* \cong k[x_1,\dots, x_n]$$ where we fixed a basis $\{x_1,\dots,x_n\}$ of $V^*$.  The Gr\"obner fan $GF(I)$ of $I$ is defined as follows. Define an equivalence relation on $\bR^n$ by $w \sim w'$, $w,w' \in \bR^n$ if $\ii_{\prec_w}I = \ii_{\prec_{w'}}I$, where $\prec_w$ is the weight order corresponding to $w$ (we use the lexicographic order for tie-breaking). Let $[w]$ be the equivalence class containing $w$. Then $GF(I)$ is the collection of $\bar{[w]}$, where $\bar{[w]}$ is the closure of $[w]$ in $\bR^n$, and these constitute a fan \cite[Proposition~2.4]{Sturmfels}. The maximal cones of $GF(I)$ are in one-to-one correspondence with the reduced Gr\"obner bases of $I$.

We apply Proposition~\ref{P:gstv} to $G = \GL(V)$ acting naturally on $\bigwedge^{P(m)}S^mV^*$. Let $R$ be the maximal torus of $G$ diagonalized by $x_i$.  Let $P(m)$ denote the Hilbert polynomial of $I$. Given $\zeta_m := \wedge^{P(m)}I_m \in \bP(\bigwedge^{P(m)}S^mV^*)$, the state polytope $\stp_{\zeta_m}(R)$ is precisely the $m$th state polytope $St_m(I)$ in \cite{Sturmfels}. By Proposition~\ref{P:gstv}, there is a Zariski dense open subscheme $U_m \subset G$ such that $St_m(I)$ remains unchanged as long as $g \in U_m$.

Sturmfels defines  state polytope $St(I)$ of $I$ to be the Minkowski sum $St(I) = \sum_{m=2}^{D}St_m(I)$, where $D$ is the largest degree of any element in a minimal universal Gr\"obner basis of $I$ \cite[(2.8)]{Sturmfels}. Then we have a Zariski open dense subscheme $\cap^D U_m$ such that $St(g.I)$ remains the same as long as $g \in \cap^DU_m$. Now, since the normal fan of the state polytope $St(I)$ is the Gr\"obner fan $GF(I)$ of $I$ \cite[Theorem~2.5]{Sturmfels}, we have:

\begin{prop}\label{P:ggf}  There exists a Zariski open dense subscheme $U \subset G$ such that $GF(g.I)$ remains the same for any $g \in U$.
\end{prop}


The Gr\"obner fan $GF(g.I)$, $g \in U$, is called  the {\it generic Gr\"obner fan} of $I$. The existence of the generic Gr\"obner fan is proved in \cite{Romer}. Roughly speaking, they construct a universal Gr\"obner basis $\{h_1(y_{ij}),\dots,h_s(y_{ij})\}$ in $k'[x_1,\dots,x_n]$ using Buchberger algorithm, where $k' = k(y_{ij} : i,j=1,\dots,n)$, in such a way that substituting $g = (g_{ij}) \in U \subseteq \GL(V)$ for $(y_{ij})$ gives a universal Gr\"obner basis of $(g.I)$ with the same support for all $g \in U$. Here, the dense Zariski open subset $U$ is determined by the non-vanishing locus of the numerators and the denominators of the polynomials of $y_{ij}$ occurring in the calculations of the algorithm.

\section{Search for Kempf's worst one parameter subgroups}\label{S:algorithm}

We first recall Mumford's numerical criterion: Let a reductive group $G$ linearly act on a projective variety $X \subset \bP(V)$ via a representation $\phi : G \to \GL(V)$. Then $x \in X$ is stable (resp. semistable) if and only if for every 1-PS $\rho$ of $G$,  the \emph{Hilbert-Mumford index} $\mu(x,\rho)$ is positive (resp. non-negative), where $\mu(x,\rho)$ is defined  as the character with which $G$ acts on the fibre $\cO_X(1)|_{x^*}$ over the fixed point $x^* = \lim_{t\to 0} \rho(t).x$. In terms of the homogeneous coordinates $\{x_0,\dots,x_n\}$ diagonalizing the $\rho$-action, i.e. $\rho(t).x_i = t^{r_i}x_i$,
\[
\mu(x,\rho) = \max\{-r_i \, | \, x_i \ne 0\}.
\]
Let $R$ be a maximal torus of $G$, $T$ a maximal torus of $\GL(V)$ containing $\phi(R)$, and assume that $T$ is diagonalized by $x_i$. Let $\chi_i \in X(T)$  be the character picking off the $i$th factor, and by abusing notation we let $\chi_i$ denote the induced character of $R$. Take an affine point $v\in V$ over $x$ and identify $\{i \, | \, x_i \ne 0\}$ with the state $\Xi_v(T)$. Then the Hilbert-Mumford index $\mu(x,\rho)$ equals
\[
\max \{-\langle\chi,\rho\rangle \, | \, \chi \in \Xi_v(T) \} = - \min\{\langle\chi, \rho\rangle \, | \, \chi \in \Xi_v(R)\}.
\]
Let $\Gamma(R)$ be the group of 1-PS of $R$. Fix a positive definite integral bilinear form on $\Gamma(R)$ that is invariant under the Weyl group of $G$ with respect to $R$.  The Killing form is an example. Let  $||\cdot ||$ be the corresponding length function on  $\Gamma(G)$. Kempf proved that if $v$ is unstable, there exists a 1-PS $\rho_o$ such that
\[
\mu(v,\rho_o)/||\rho_o|| \le \mu(v,\rho)/||\rho||
\]
for any nontrivial $\rho \in \Gamma(G)$ (a conjecture of Mumford).

To demonstrate the relation between the numerical criterion and the state polytope, we briefly review a duality theorem which can be proven by elementary means. This is well known in approximation theory. Let $\bE$ be a finite dimensional vector space and $(\, , \,)$ be a positive definite real valued  symmetric bilinear form on $\bE$. Let $\Xi = \{ \chi_1, \dots, \chi_s\}$ be a finite subset of $\bE$ and $P$ denote its convex hull. Let $|\chi| = (\chi,\chi)^{1/2}$ denote the associated norm and $|\cdot|^*$, the dual norm on $\bE^*$: $|\rho|^* = \max\{\langle\chi,\rho\rangle \, | \, |\chi|\le 1\}$. Choose a suitable basis (and its dual) with respect to which the pairing $\langle\chi,\rho\rangle$ is the matrix multiplication $\chi^T\rho$. We also use the basis to fix an isomorphism between $\bE$ and $\bE^*$. The duality is between the following two problems:
\begin{enumerate}
	\item[(A)] Find $\min\{|\chi| \, | \, \chi \in P\}$;
\item[(B)] Find $\max\{g(\rho) \,  | \, |\rho|^* \le 1\}$, where $g(\rho) = \min_{\chi\in \Xi} \langle\chi, \rho\rangle$.
\end{enumerate}
That is, we have:

\begin{thm} \label{T:nearest}  Suppose that $\{\rho \, | \, g(\rho) \ge 0\}$ is nonempty.
If $\chi_o$ is the solution for the first problem and $\chi_o \ne 0$, then the solution for the second is given by $\chi_o/|\chi_o|$ (via the chosen isomorphism $\bE \simeq \bE^*$).
\end{thm}

For a proof of this theorem, see,  for instance, \cite{Wolfe}.  Back to the discussion on the worst 1-PS.
Here, $\bE = X(R)\otimes_\bZ\bR$, $\bE^* = \Gamma(R)\otimes_\bZ\bR$ and an isomorphism between them is given by the nondegenerate invariant inner product (for instance, again, the Killing form) on $\bE$. Set $\Xi = \Xi_v(R)$. In this case, the functional $g(\rho)$ on $\Gamma(R)_\bR$ is precisely  $\mu(v, \rho)$ and the second problem (B) amounts to finding a worst 1-PS of $R$. By Theorem~\ref{T:nearest}, this is equivalent to finding the point in the state polytope $\stp_v(R)$ closest to the origin. By now there are various algorithms for finding the nearest point, and MatLab has an implementation of it, for instance. To find a worst 1-PS of $G$, one has to take into consideration the state polytopes $\stp_v(R)$ for every maximal torus $R$, and this is where our generic state polytopes may play a role. In short, the strategy is to compute all generic state polytopes $\gstp_{R,S}(v)$ and compute the nearest point for each of them. We shall prove in a later section that the primary generic state polytope always contains the origin (Theorem~\ref{T:generic-stability}).

\begin{enumerate}
\item For each nonempty $S \subset P(V)$, find $g_S$ such that $h_{\chi,\a}^v(g_S)$ is zero for all $\chi \in S$ and nonzero otherwise. This is in general computationally extremely difficult. It is also the one with most room for exploration by using, for instance, representation theoretic approaches;
\item Compute the nearest point $\rho_S$ of $\stp_{g_S.v}(R)$ for each $S$;
\item Find the farthest point $\rho^{\star}$ among $\rho_S$'s. The length of $\rho^{\star}$ is precisely the negative of the value of the worst Hilbert-Mumford index.
\end{enumerate}

\section{Finiteness of destabilizing 1-PS}
Let $G$ and $V$ be as before and let $v\in V$ be an unstable point. If $G$ is semisimple,  a general maximal torus never contains a destabilizing 1-PS by Theorem~\ref{T:generic-stability}. By the Hilbert-Mumford criterion, there exists a 1-PS $\rho$ of a maximal torus $R \subset G$ which pairs positively with every character $\chi$ of $R$ such that $v_\chi \ne 0$. That is, the intersection $C_v(R)$ of all half spaces $\{\rho \in \Gamma(R)_\bR \, | \, \langle \chi, \rho \rangle > 0\}$, $\forall \chi \in \Xi_v(R)$, is nonempty. Since $\Xi_v(R)$ is finite,  $C_v(R)$ is a rational polyhedral cone. Hence
\begin{lem}
The set of 1-PS of $R$ that destabilizes $v$ is $C_v(R)\cap \Gamma(R)$, which is generated by a finite set $D_v(R) \subset X(R)$.
\end{lem}

\begin{lem}\label{L:destab-1ps}
$D_v(g^{-1}Rg) = (g^!)^{-1}D_{g.v}(R)$
\end{lem}

\begin{proof} Suppose that $\rho \in \Gamma(g^{-1}Rg)$ destabilizes $v$ i.e.,  $\langle \chi, \rho \rangle > 0 $ for all $\chi \in \Xi_v(g^{-1}Rg)$. From the basic relation (see Section~\ref{S:generic-semistability}, Equation~(\ref{E:basic-reln}))
\[
g.v_\chi = (g.v)_{g_!\chi}, \quad \forall \chi \in X(g^{-1}Rg), \quad g_! : X(g^{-1}Rg) \stackrel{\sim}{\to} X(R)
\]
and the obvious $G$-invariance of the pairing
\[
\langle \chi, \rho \rangle = \langle g_!\chi, g^!\rho \rangle, \quad  \forall \rho \in \Gamma(g^{-1}Rg), g^! : \Gamma(g^{-1}Rg) \stackrel{\sim}{\to} \Gamma(R)
\]
we see that if $\rho \in \Gamma(g^{-1}Rg)$ destabilizes $v$ (i.e. $\langle \chi, \rho \rangle > 0$ for some $\chi \in \Xi_v(g^{-1}Rg)$) then $g^!\rho \in \Gamma(R)$ destabilizes $g.v$.
\end{proof}
In the discussion preceding this section, we have described how one can compute a finite set  $\{g_1, \dots, g_s\} \subset G$ that  generates all state polytopes of $v$ i.e.
\[
\bigcup_{g\in G}g_*\Xi_v(g^{-1}Rg) = \bigcup_{g\in G}\st_{g.v}(R) = \bigcup_{i=1}^s \st_{g_i.v}(R).
\]
Combining this with the Lemma~\ref{L:destab-1ps} above, we see that
\begin{prop}\label{P:destab-1ps}
The set of {\it all} destabilizing 1-PS of $v$ is generated by the finite set
\[
\coprod_{i=1}^s (g_i^!)^{-1}D_{g_i.v}(R) \subset \Gamma(G).
\]
\end{prop}

\section{Generic semistability}\label{S:generic-semistability} 

Let $G$ be a reductive algebraic group and $\rho: G \to \GL(V)$ be a rational representation.

\begin{defn} A point $v \in V$ is said to be {\it generically semistable (stable)} if the (interior of the) primary generic state polytope  $\gstp_v(R)$ contains the trivial character.
\end{defn}

In other words, generic semistability means $R'$-semistability for general maximal tori $R'$. When $G$ is not reductive, we don't have the numerical criterion and the  $R$-semistability for all tori $R$ does not imply $G$-semistability, so we restrict ourselves to the case when $G$ is reductive.

Even in a worst GIT problem in which every point is unstable, one needs to choose a maximal torus carefully to produce a destabilizing 1-ps:

\begin{exmp} (Hyperplanes of $\bP^n$) Let $G = \SL_{n+1}(\bC)$ naturally act on the space $\bP H^0(\cO_{\bP^n}(1))$ of hyperplanes of $\bP^n$. In this case, the semistability with respect to generic maximal tori is equivalent to the semistability in generic coordinates.

A hyperplane $H$ in generic coordinates $x_0,\dots,x_n$  is $\sum_{i=0}^n a_i x_i$ with $a_i \ne 0$ for any $i$, and is stable with respect to the maximal torus of $G$ diagonalized by $\{x_0,\dots, x_n\}$. With respect to a coordinate system such that $H = \{y_0 = 0\}$, $H$ is destabilized by the obvious 1-ps with weights $(n, -1, -1, \dots, -1)$.
\end{exmp}

In this section, we prove that in {\it any} GIT problem i.e. for any connected complex semisimple group $G$ and a rational representation $V$ of $G$ and any nonzero point $v \in V$, a general torus does not contain a destabilizing 1-ps for $v$. The key point is the well known fact that for any complex semisimple Lie algebra and its representation, there is a sufficient symmetry of the weights.

We shall use the Lie algebra $\mathfrak g$ of $G$, and the correspondence between $\mathfrak g$-modules and $G$-modules.
Let $G$ be a connected complex semisimple algebraic group with Lie algebra $\mathfrak g$, fix a maximal torus $R$ and let $\mathfrak r$ denote its Lie algebra. The adjoint representation $Ad : G \to \GL(\mathfrak g)$ induces the weight space decomposition
\[
\mathfrak g = \mathfrak r \oplus \bigoplus_{\a\in \Phi}\mathfrak g_\a.
\]
Then it is well known that
\begin{enumerate}
\item $\mathfrak g$ is a semisimple Lie algebra;
\item $(X(R)_\bR, \Phi, X(R)^\vee_\bR, \Phi^\vee)$ is a root system, where $X(R)^\vee$ is the group of co-characters and $\Phi^\vee$ is the set of co-roots with respect to the natural pairing $\langle \quad, \quad \rangle$ of characters and co-characters;
\item The root system above is precisely the root system of the semisimple Lie algebra $\mathfrak g$ with respect to the Cartan subalgebra $\mathfrak t$, and;
\item $\weyl : = N_G(R)/R$ is finite, acts on $X(R)_\bR$ by conjugation and is naturally identified with the Weyl group  of the root system.
\end{enumerate}

Let   $\phi: \mathfrak g \to \mathfrak gl(V)$ denote the Lie algebra representation induced by $\rho: G \to \GL(V)$.
The weight space decomposition $V = \oplus_{\chi \in X(R)} V_\chi$ coincides with the decomposition of $V = \oplus_{d\chi \in \mathfrak t^*} V_{d\chi}$ of the weight decomposition of $V$ as a Lie algebra representation of $\mathfrak g$ i.e. $V_\chi = V_{d\chi}$.
We recall a few fundamental results from the representation theory of Lie algebras (see \cite{Varadarajan}, eg):

\begin{thm} The Weyl group $\weyl$ acts on the set of weights of $V$, and simply transitively on the set of the Weyl chambers.
\end{thm}

The following is immediate from the theorem above:

\begin{coro}\label{C:Weyl} The convex hull in $X(R)_\bR$ of the $\weyl$-orbit of a  nonzero weight of $V$ contains the origin.
\end{coro}

We recapitulate the following basic facts for readers' convenience. Recall that an algebraic group is said to be {\it almost simple} if it is smooth, connected, non-commutative and every proper normal subgroup of it is finite. An algebraic group $G$ is said to be an {\it almost direct product} of its subgroups $G_1, \dots, G_r$ if the multiplication map $G_1 \times \cdots \times G_r \to G$ is surjective with finite kernel. Then an algebraic group is semisimple if and only if it is an almost direct product of its almost simple subgroups (called its {\it almost simple factors}).

\begin{thm}\label{T:generic-stability} Let $G$ be a semisimple algebraic group and $\rho : G \to \GL(V)$ be a rational representation. Then
\begin{enumerate}
\item any nonzero point $v \in V$ is generically semistable, and;
\item it is generically stable if and only if no non-trivial almost simple factor of $G$ is contained in the isotropy subgroup $G_v$.
\end{enumerate}
\end{thm}
\begin{proof} 1. \,
Let $\chi$ be a weight of $v$ such that $v_\chi \ne 0$.
Define
\[
U_\chi := \{g \in G \, | \, \chi \in \Xi_{g.v}(R)\}.
\]
By definition it equals $ \{g \in G \, | \, (g.v)_\chi \ne 0\}$, and hence is  an open subvariety of $G$ (Section~\ref{S:gst}).
Recall the isomorphism
\[
g_! : X(g^{-1}Rg) \to X(R)
\]
given by $(g_!\chi)(t) = \chi(g^{-1}t g)$. For $g \in N_G(R)$, $g_!$ gives an automorphism of $X(R)$, and we have
\begin{equation}\label{E:basic-reln}
t.(g.v_\chi) = g.(g^{-1}tg).v_\chi = (g_!\chi)(t) g.v_\chi  \tag{$\dagger\dagger$}
\end{equation}
for any $t\in R$ and $g \in N_G(R)$. That is, $g.v_\chi = (g.v)_{g_!\chi}$ and $\chi \in \Xi_v(R)$ if and only if $g_!\chi \in \Xi_{g.v}(R)$.

Let $w \in \weyl = N_G(R)/R$ and $g \in N_G(R)$ be a point over $w$. If $\chi \in \Xi_v(R)$ then $w(\chi) = g_!\chi \in \Xi_{g.v}(R)$. So for any  $w \in \weyl$ and a state $\chi \in \Xi_v(R)$, $U_{w(\chi)}$ contains a representative in $N_G(R)$ of $w^{-1}$ and hence is a nonempty open subvariety of $G$. Set $U = \bigcap_{w \in W}U_{w(\chi)}$. For any $g \in U$, it follows that $(g.v)_{w(\chi)} \ne 0$ for all $w \in \weyl$, and the $\Xi_{g.v}(R)$ contains a $\weyl$-orbit. By Corollary~\ref{C:Weyl}, the generic state polytope of $v$ contains the origin.

\

\noindent 2. \, Let $G_1, \dots, G_r$ be almost simple factors of $G$. Any maximal torus $R$ of $G$ is an almost direct product of maximal tori $R_i$ of $G_i$, and the Weyl group $\mathcal W(G,R)$ is the direct product of $\mathcal W_i := \mathcal W(G_i, R_i)$.

Suppose first that  $G_i$ acts trivially on $v$, for some $i$. Then any state $\lambda$ of $v$ satisfies $\lambda(R_i) = 1$, so that the state polytope $\mathcal P_v(R) $ is contained in a proper hyperplane. Hence the generic state polytope has empty interior, and $v$ is generically strictly semistable.

Conversely, suppose that no $G_i$ is contained in $G_v$.  We claim that each $G_i$ contains a maximal torus that is not contained in $G_v$: Let $H$ be the subgroup of $G_i$ generated by all maximal tori of $G_i$. If $G_v$ contains all maximal tori, then $G_v$ contains $H$. But $H$ is normal and connected so $H = G_i$ since $G_i$ is almost simple. Moreover, since $G_v$ is a closed subgroup, it follows that a general maximal torus of $G_i$ acts non-trivially on $v$.

Fix a maximal torus $R \subset G$ whose simple factors $R_i \subset G_i$ act non-trivially on $v$.  Let $\mathfrak g_i = \mathfrak r_i \oplus \bigoplus_{\alpha \in\Phi_i} \mathfrak g_{i, \alpha}$ be the root space decomposition. Now, by the first part of the theorem, there exists a nonempty open subscheme $U \subset G$ such that for any $g \in U$, the state polytope of $g.v$ with respect to $R$ contains the origin. We have observed that a general maximal torus of each simple factor $G_i$ acts non-trivially on $v$. By Equation~(\ref{E:basic-reln}), we may assume that each $R_i$ acts nontrivially on $g.v$ since $g$ is chosen generally. Hence there exists a state $\lambda_i$ of $g.v$ that is non-trivial on $\mathfrak r_i$. Since $G_i$ is almost simple, its root system is indecomposable and the Weyl orbit $\mathcal W_i.\lambda_i$ spans $\mathfrak r_i^*$. Said otherwise, the convex hull of $\mathcal W_i.\lambda_i$ in $\mathfrak r_i^*$ contains the origin in its interior, for any $i$.  Hence the state polytope of $g.v$ with respect to $R$ contains the origin in its interior.

\end{proof}

\section{Generic stability of Hilbert points}
In this section, we take a closer look at the case of Hilbert points, our main objects of interest.
Let $\hilb^P \bP^n$ denote the Hilbert scheme of closed subschemes of Hilbert polynomial $P$ of $\bP^n$  and let $S = k[x_0,\dots,x_n]$ be the homogeneous coordinate ring of $\bP^n$.
For $m\ge m_0$, the Gotzmann number of $P$, we have the Grothendieck embedding
\[
\hilb^P \bP^n \inj \Gr(\ell, S_m) \inj \bP\left(\bigwedge^\ell S_m\right)
\]
where $S_m$ is the degree $m$ piece of the graded algebra $S$ and $\ell = \binom{n+m}m - P(m)$. The embedding sends a saturated ideal $I \subset S$ of Hilbert polynomial $P$ to $I_m \subset S_m$ which is an $\ell$ dimensional linear subspace. Let $V := \bigwedge^\ell S_m$.


We have a natural linear action of $\SL_{n+1}(k)$ on $V$ which preserves $\hilb^P\bP^n$. Let $T$ be the maximal torus of $\SL_{n+1}(k)$ diagonalized by $x_0,\dots,x_n$. The weight spaces of $V$ are simply $k.x^\a := k.(x^{\a_1}\wedge \cdots \wedge x^{\a_\ell})$ where $x^{\a_i} = \prod_{j=0}^n x^{\a_{ij}}$: Any $t \in T$ acts on $x^\a$ via the character
\[
\chi_\a(t) = \prod_{i=1}^\ell \prod_{j=0}^n t_j^{\a_{ij}} = \prod_{j=0}^n t_j^{\sum_{i=1}^\ell \a_{ij}}.
\]
Different $\a$s may very well give rise to the same character. In fact, $\chi_\a$ is trivial if and only if $\sum_{i=1}^\ell \a_{ij}$ is independent on $j$. That is, the columns of the $\ell \times (n+1)$ matrix $(\a_{ij})$ have the same column sum. Hence the non-triviality of the trivial weight space is equivalent to the existence of a non-negative integral $\ell\times (n+1)$ matrix $(\a_{ij})$ such that
\begin{enumerate}
\item the rows $\a_1, \dots, \a_\ell$ are all mutually distinct;
\item the row sums $\sum_{j=0}^n \a_{ij}$ equal $m$;
\item the column sums $\sum_{i=1}^\ell \a_{ij}$ are the same.
\end{enumerate}
There are obvious necessary conditions.  If $(\a_{ij})$ satisfied the column/row sum conditions, summing up all entries of $(\a_{ij})$, we would obtain
\[
\ell m = (n+1) c
\]
where $c$ denotes the column sum. This implies that $(n+1) | \ell$. Also, for the rows to be distinct,
$\ell$ should not exceed $\binom{n+m}m$, the number of all degree $m$ monomials in $x_0, \dots, x_n$. These conditions turn out to be sufficient as well, but we do not present  a proof here.
Back to the generic stability problem: We first examine the hypersurface case. Let $G = \SL_{n+1}(k)$ act on the space $H^0(\bP^n, \cO_{\bP^n}(d))$ of degree $d$ homogeneous polynomials in $x_0, \dots, x_n$. Let $T$ be the maximal torus of $G$ diagonalized by $x_0, \dots, x_n$.  The weight spaces are precisely the one-dimensional spaces $k.x^\a$ where $x^\a$ are the degree $d$ monomials, and we will conflate monomials and their associated characters. Note that when $n = 1$ and $d$ is odd, the trivial weight space is zero. Now, consider the GIT of $\bP H^0(\cO_{\bP^n}(d))$.

\begin{prop} For any $[f] \in \bP H^0(\cO_{\bP^n}(d))$, the generic state is
\[
\gst_{[f]}(T) = \left\{ x^\a \, \left| \, \forall \a \in \bZ^{n+1}_{\ge 0}, \sum_{i=0}^n \a_i = d\right\}\right.. \
\]
\end{prop}

\begin{proof} Let $g=(g_{ij})\in G$. Then $g.x_i=\displaystyle{\sum_{j=0}^{n}}g_{ji}x_j$. Thus for a degree $d$ monomial $x^{\alpha}=x_0^{\alpha_0}\ldots x_n^{\alpha_n}$, we have
$$g.x^{\alpha}=\displaystyle{\sum_{\deg(\beta)=d}} c_{\beta}^{\alpha}(g)x^{\beta}.$$
Substituting  $g.x_i=\displaystyle{\sum_{j=0}^{n}}g_{ji}x_j$ and expanding, we see that $c_{\beta}^{\alpha}(g)$ is a nonzero polynomial in $\mathbb{Z}[g_{ij}]$ with positive integer coefficients, and that it is a multi-homogeneous polynomial of multi-degree $(\a_0, \dots, \a_n)$ in the $(n+1)$ sets of variables $$(g_{00},g_{10},\dots,g_{n0}; g_{01}, \dots, g_{n1}; \cdots ; g_{0n},\dots,g_{nn}).$$
Now let $f=\displaystyle{\sum_{\deg(\alpha)=d}}d_{\alpha}x^{\alpha}$ be a homogeneous polynomial of degree $d$, and consider
\[
g.f=\displaystyle{\sum_{\deg(\beta)=d}}\left(\displaystyle{\sum_{\alpha}}d_{\alpha}c_{\beta}^{\alpha}(g)\right)x^{\beta}.
\]
No cancellation occurs in $\displaystyle{\sum_{\alpha}}d_{\alpha}c_{\beta}^{\alpha}(g)$ since the coefficients $c^{\a}_\b(g)$ have different multi-degrees. Therefore, the state of $g.f \in  H^0(\cO_{\bP^n}(d))$ contains all weights of $H^0(\cO_{\bP^n}(d))$ but especially (the characters associated to) $x_0^d, \dots, x_n^d$, whose convex hull clearly contains the trivial character.
\end{proof}

\begin{prop} \label{P:gen-stab-hilb}
Let $I \subset k[x_0,\dots,x_n]$ be a homogeneous ideal. The $m$th Hilbert point $[I]_m$ is generically stable with respect to the natural $\SL_{n+1}(k)$-action,  for any $m \ge \textup{reg}(I)$.
\end{prop}
This of course follows from the main Theorem~\ref{T:generic-stability} applied to the $\SL_{n+1}(k)$ action on $\bigwedge^{\dim I_m}k[x_0,\dots,x_n]_m$, but here we present a proof devoid of any Lie algebra terminologies.
\begin{proof}
Let $f_1,\ldots,f_l$ be a $k$-basis of $[I]_m$ and let $g=(g_{ij})\in \SL_{n+1}$. Then
$$g.(f_1\wedge \cdots \wedge f_l)=g.f_1\wedge \cdots \wedge g.f_l=
\sum c_{\alpha}(g)x^{\alpha (1)}\wedge\cdots\wedge x^{\alpha (l)}$$
where $c_{\alpha}(g)$ is a polynomial in $\mathbb{Z} [g_{ij}]$ and  $x^{\alpha (i)}=x_0^{\alpha (i)_0}\cdots x_n^{\alpha (i)_n}$ are degree $m$ monomials.

Note that the Weyl group of $\SL_{n+1}(k)$ is isomorphic to the symmetric group $S_{n+1}$ on $n+1$ letters.
For  $s\in S_{n+1}$ (here we identify $S_{n+1}$ with  a natural copy of it in $\SL_{n+1}$),
the action of $s$ on $f_1\wedge \cdots \wedge f_l$ just permutes  $x_0,\ldots,x_n$ and the corresponding coefficients are unchanged (invariance of coefficients).

Define an $S_{n+1}$ action on $\alpha (i)$ and $\alpha$ by
$$s.x^{\alpha (i)}=x_{s(0)}^{\alpha(i)_0}\ldots x_{s(n)}^{\alpha(i)_n}=x_0^{(s.\alpha(i))_0}\ldots x_n^{(s.\alpha(i))_n}$$
and
$$s.(x^{\alpha(1)}\wedge \ldots \wedge x^{\alpha(l)})=x^{s.\alpha(1)}\wedge \ldots \wedge x^{s.\alpha(l)}.$$
Then if we let $U_{\alpha}=\{g=(g_{ij}) \, \vert \, c_{\alpha}(g)\ne 0\}$ and choose any $\alpha$ with $U_{\alpha}\ne \phi$, each $U_{s.\alpha}$ is also nonempty open precisely due to the invariance of coefficients. Since $\GL_{n+1}$ is irreducible, $U=\bigcap_{s\in S_{n+1}}U_{s.\alpha}$ is nonempty open. Then for any $g\in U$, $g.(f_1\wedge \ldots \wedge f_l)$ is clearly stable.
\end{proof}

\begin{remark} In the analysis above, Hilbert points are generically stable since the stability is with respect to an action of $\SL_{n+1}(k)$ which is almost simple: $\SL_{n+1}(k)/\mathbb{\mu}_{n+1} = \PSL_{n+1}(k)$ has no nontrivial normal subgroup.

\end{remark}

\end{document}